\documentclass[english]{ourlematema}
\usepackage{amsmath, amsthm, amssymb, hyperref, color}
\usepackage{MnSymbol} 
\usepackage{graphicx}
\usepackage{caption}
\usepackage[all]{xypic}
\usepackage{verbatim}
\usepackage{chemfig, chemnum}
\usepackage{tikz}
\usepackage[linesnumbered,lined,commentsnumbered]{algorithm2e}
\usepackage{caption}
\usepackage{subcaption}
\usepackage{makecell}
\usepackage{array}

\newtheorem{theorem}{Theorem}
\numberwithin{theorem}{section}
\newtheorem{proposition}[theorem]{Proposition}
\newtheorem{lemm}[theorem]{Lemma}
\newtheorem{corollary}[theorem]{Corollary}

\newtheorem{remark}[theorem]{Remark}
\newtheorem{example}[theorem]{Example}
\newtheorem{conjecture}[theorem]{Conjecture}

\theoremstyle{definition}

\newcommand{\PP}{\mathbb{P}}
\newcommand{\RR}{\mathbb{R}}

\newcommand{\CC}{\mathbb{C}}

\newcommand{\NN}{\mathbb{N}}
\newcommand{\SSS}{\mathbb{S}}

 \title{Pencils of Quadrics: Old and New}
 
  \author{Claudia Fevola}
  \address{%
  MPI for Mathematics in the Sciences, Leipzig \\
\email{claudia.fevola@mis.mpg.de}
}

\author{Yelena Mandelshtam}
\address{%
University of California, Berkeley \\
\email{yelenam@berkeley.edu}
}

\author{Bernd Sturmfels}
\address{%
  MPI for Mathematics in the Sciences, Leipzig \\
\email{bernd@mis.mpg.edu }
}

\date{2020/09/09}

\begin{document}
\maketitle
\begin{abstract}
\noindent
 Two-dimensional linear spaces of symmetric matrices 
are classified by Segre symbols.  
After reviewing known facts from
linear algebra and projective geometry, we address new questions
motivated by algebraic statistics and optimization.
We compute the reciprocal curve and the maximum likelihood degrees,
and we study strata of pencils in the Grassmannian.
\end{abstract}

\section{Introduction}

A pencil of quadrics is a two-dimensional 
linear subspace $\mathcal{L} $ in 
 the space $ \SSS^n$ of (real or complex) symmetric $n \times n$ matrices.
It is a point in the Grassmannian ${\rm Gr}(2,\SSS^n)$, and it
specifies a line  $\PP \mathcal{L}$  in the projective space
 $\PP(\SSS^n) \simeq \PP^{\binom{n+1}{2}-1}$.
The group ${\rm GL}(n)$ acts on $\SSS^n$ by congruence and this
induces an action on ${\rm Gr}(2,\SSS^n)$. We say that two 
pencils are {\em isomorphic} if they lie in the same ${\rm GL}(n)$-orbit.

Fix a pencil $\mathcal{L}$ with basis $\{A,B\}$. The determinant
$\,{\rm det}(\mathcal{L}) \,=\, {\rm det}(\lambda A+ \mu B) \,$
 is well-defined up to the action of ${\rm GL}(2)$ by changing basis in $\mathcal{L}$.
The zeros of this binary form are a multiset of size $n$ in the line $\PP^1$,
well-defined  up to isomorphism of  $\PP^1$.
We exclude pencils $\mathcal{L}$ that are {\em singular}, meaning that ${\rm det}(\mathcal{L}) = 0 $.
The singular pencils form a subvariety ${\rm Gr}(2,\SSS^n)^{\rm sing}$ in the Grassmannian. 
We are interested in a natural stratification
of the open set of all regular pencils:
$$ {\rm Gr}(2,\SSS^n)^{\rm reg}\,\,\, = \,\,\,
 {\rm Gr}(2,\SSS^n) \mathbin{\big\backslash}   {\rm Gr}(2,\SSS^n)^{\rm sing}. $$
Each stratum is indexed by a {\em Segre symbol} $\sigma$. This is
a multiset of partitions
whose parts add up to $n$ in total. One exception:
the singleton $[ (1,1,\ldots, 1) ]$ is not a Segre symbol.
The number $S(n)$ of Segre symbols 
was already of interest to Arthur Cayley in 1855. 
In \cite[p.~316]{cayley}, he derived the generating function
$$ \sum_{n=1}^\infty  S(n) x^n \,\, = \,\, 
\prod_{k \geq  1}\frac{1}{(1-x^{k})^{P(k)}} \,- \,\frac{1}{1{-}x}  \,\, = \,\,
2x^2+5x^3+13x^4+26x^5+57x^6+110x^7 +\, \cdots, $$
where $P(k)$ is the number of partitions of the integer $k$.
The two Segre symbols for $n=2$ are
$[1,1] $ and $[2]$.
For $n=3$ and $n=4$ they
are shown in Figure~\ref{fig:segre34}.

The Segre symbol  $\sigma = \sigma (\mathcal{L})$
of a given pencil $\mathcal{L}$ can be computed as follows.
Pick a basis $\{A,B\}$ of $\mathcal{L}$, where $B$ is invertible,
and find the Jordan canonical form of $AB^{-1}$.
Each eigenvalue of $AB^{-1}$ determines a partition,
according to the sizes of its Jordan blocks.
Then $\sigma$ is the associated multiset of partitions.
It turns out that $\sigma$ does not depend on the choice of basis $\{A,B\}$.
For  the relevant background  in linear algebra see
\cite{DE, thompson, uhlig76} and  Section~\ref{sec2} below.

The role of Segre symbols in projective geometry can
be stated as follows.

\begin{theorem}[Weierstrass-Segre] \label{thm:WS}
Two pencils of quadrics in $\SSS^n$ are isomorphic if and only if  their
Segre symbols agree and their determinants define the same multiset 
of $n$ points on the projective line $\PP^1$,  up to isomorphism of $\PP^1$.
\end{theorem}

\begin{example}[$n=2$] \label{ex:n2}
All pencils $\mathcal{L}$ are regular. There are two ${\rm GL}(2)$-orbits, given
by the rank of a matrix $X$ that spans  
$\,\mathcal{L}^\perp = \{X \in \SSS^2: {\rm trace}(AX) = {\rm trace}(BX) = 0\} $.
If $X$ has rank $2$ then ${\rm det}(\mathcal{L})$ has two distinct roots in $\PP^1$ and
the Segre symbol is
 $\sigma(\mathcal{L})=[1,1] $. If $X$ has rank $1$ then
it is a double root in $\PP^1$ and 
$\sigma(\mathcal{L})=[2] $.
\end{example}

We learned about Theorem  \ref{thm:WS} from an unpublished note
by Pieter Belmans, titled  {\em Segre symbols}, which credits
the 1883 PhD thesis of Corrado Segre. It appears in the
textbooks on algebraic geometry   by Dolgachev \cite[\S 8.6.1]{Dol}
and Hodge-Pedoe \cite[\S XIII.10]{HP}.
The idea goes back to at least the 1850s, in works of
Cayley  \cite{cayley}
and Sylvester \cite{sylvester}.
One aim of this article is to revisit this history.

We begin in Section \ref{sec2} with a linear algebra perspective on Theorem \ref{thm:WS},
with focus on normal forms for pencils.
We denote by $\mathcal{L}^{-1}$ the set of the
  inverses of all invertible matrices in $\mathcal{L}$. Since we
  exclude singular pencils, this set is nonempty. Its closure in
   $\PP(\SSS^n)$ is a projective curve, called the \textit{reciprocal curve} and  denoted 
    $\PP \mathcal{L}^{-1}$.
In Section \ref{sec3} we study the reciprocal curve $\PP \mathcal{L}^{-1}$ of
a  pencil $\mathcal{L} \in {\rm Gr}(2,\SSS^n)^{\rm reg}$.
This curve is parametrized by
the inverses of all invertible matrices in $\mathcal{L}$.
We prove that $\PP \mathcal{L}^{-1}$ is a rational normal curve.
We express its degree in terms of the Segre symbol $\sigma(\mathcal{L})$,
and we determine its prime ideal.

In Section \ref{sec4} we turn to maximum likelihood estimation 
for Gaussians. A linear Gaussian model is a set of multivariate Gaussian probability
distributions whose covariance or concentration matrices are linear combinations of some fixed symmetric
matrices. Hence, when restricting to two-dimensional models, a pencil $\mathcal{L}$ plays two different roles in statistics,
depending on whether it lives
in the space of concentration matrices (as in \cite{SU})
 or in the space of covariance matrices (as in \cite{CMR}).
 This yields two numerical invariants, the ML degree 
 ${\rm mld}(\mathcal{L})$  and the reciprocal ML degree
  ${\rm rmld}(\mathcal{L})$.
    We compute these in Theorem~\ref{thm:sec4main}.

In Section \ref{sec5} we study the constructible set defined by a fixed Segre symbol:
\begin{equation}
\label{eq:stratum}  {\rm Gr}_\sigma \,\, = \,\,
\bigl\{ \,\mathcal{L} \in {\rm Gr}(2,\SSS^n)^{\rm reg}\,\,: \,\, \sigma(\mathcal{L}) = \sigma \,\bigr\}. 
\end{equation}
Its closure $\overline{{\rm Gr}}_\sigma$ is a variety. We study these varieties and
their poset of inclusions, seen in Figure \ref{fig:segre34}.
 This extends the stratification of ${\rm Gr}(2,\RR^n)$ by matroids, see \cite{gelfand}.
 Indeed, if $\mathcal{L}$ consists of diagonal matrices then
the Segre symbol $\sigma(\mathcal{L})$ specifies
 the rank $2$ matroid of $\mathcal{L}$, up to permuting the ground set $\{1,2,\ldots,n\}$.

 \begin{example}[$n=3$] \label{ex:n3} There are
  five strata ${\rm Gr}_\sigma$ in
 the Grassmannian ${\rm Gr}(2,\SSS^3)$:
   $$
  \begin{small}
\begin{matrix} 
\hbox{symbol} & \! \!\! {\rm codim}\! \! & {\rm degrees} & P & Q &  \hbox{variety in $\PP^2$} \\
[1,1,1] & 0 &  (2,2,3) & ax^2{+}by^2{+}cz^2 & x^2{+}y^2{+}z^2  & \hbox{\em four reduced points} \\
[2,1] & 1 &  (2,1,2) & 2a xy {+} y^2 {+} bz^2 & 2xy+z^2  & \!\!\! \hbox{\em one double point, two others} \\
[\,3\,] & 2 &  (2 ,0,1)  &  2axz{+}ay^2{+}2yz & 2xz+y^2 & \hbox{\em one triple point, one other}\\
[(1,\!1),1] & 2 &  (1,1,1) & ax^2{+}ay^2{+}bz^2 & x^2{+}y^2{+}z^2 & \hbox{\em two double points} \\
[(2,1)] & 3 &  (1,0,0) &  2axy{+}y^2{+}az^2 & 2xy+z^2  & \hbox{\em quadruple point} \\
\end{matrix}
\end{small}
$$

 For each Segre symbol $\sigma$, we display ${\rm codim}({\rm Gr}_\sigma)$,
  the triple of degrees $\bigl({\rm deg}(\mathcal{L}^{-1}), {\rm mld}(\mathcal{L}),  {\rm rmld}(\mathcal{L})\bigr)$, the basis $\{P,Q\}$ from Section \ref{sec2}, and its variety in $\PP^2$.
Here, $x,y,z$ are coordinates on $\PP^2$, and
 $ a,b,c$ are distinct nonzero~reals.
 This accounts for all regular pencils. A pencil  is singular if $P$ and $Q$
share a linear factor. 
One such  $\mathcal{L}$ is spanned by $ xy$ and $xz$. This
defines a line and a point  in $\PP^2$. 
We conclude that  $ {\rm Gr}(2,\SSS^3)^{\rm sing}$ is an irreducible variety of  dimension~$4$. 
\end{example}

\section{Canonical Representatives}
\label{sec2}

We identify symmetric $n \times n$ matrices $A$ with quadratic forms 
${\bf x} A {\bf x}^T$ in unknowns ${\bf x} = (x_1,\ldots,x_n)$. We fix
the field to be $\CC$. The
$\binom{n+1}{2}$-dimensional vector space $\SSS^n $
is equipped with the trace inner product
 $\, (A , B ) \mapsto \mathrm{trace}(A B)$.
The group ${\rm GL}(n)$ acts on quadratic forms by linear changes of coordinates, via
${\bf x} \mapsto {\bf x} g$.
This corresponds to the  action of ${\rm GL}(n)$  on symmetric matrices by congruence:
$$ {\rm GL}(n) \times \SSS^n \,\rightarrow \,\SSS^n\,,\, \,\, (g,A ) \,\mapsto \,g A g^T . $$

Let $\mathcal{L} = \CC \{A,B\}$ be a regular pencil in ${\rm Gr}(2,\SSS^n)$,
with ${\rm det}(B) \not= 0$.  The polynomial ring
$\CC[\lambda]$ in one variable $\lambda$ is a principal ideal domain. The cokernel
of the matrix $A-\lambda B$ is a module over this PID. Consider its {\em elementary divisors}
\begin{equation}
\label{eq:elemdiv}
(\lambda - \alpha_{1})^{e_1}, \,(\lambda - \alpha_{2})^{e_2},\,\ldots\,,\, (\lambda - \alpha_{s})^{e_s}.
\end{equation}
Here $e_1,\ldots,e_s$ are positive integers whose sum equals $n$.
The list  (\ref{eq:elemdiv}) is unordered and its product is
${\rm det}(\mathcal{L}) =
\pm {\rm det}(A - \lambda B)$.
The complex numbers $\alpha_i$ are the {\em eigenvalues} of the pair $(A,B)$.
They form a multiset of cardinality $n$ in $\PP^1$.

Suppose there are $r$ distinct eigenvalues $\alpha_i$. We have $r \leq s \leq n$.
The exponents $e_i$ corresponding to one fixed eigenvalue form a partition.
This gives a multiset of $r$ partitions, with  $s$ parts in total,
where the sum of all parts is $n$. This multiset of partitions
is the Segre symbol $\sigma = \sigma(\mathcal{L})$. It is thus
visible in (\ref{eq:elemdiv}).
We now paraphrase Theorem~\ref{thm:WS} using the elementary divisors 
of the matrix $A - \lambda B$.

\begin{corollary}\label{transCoord}
Consider two quadrics ${\bf x} A {\bf x}^T$ and ${\bf x} B{\bf x}^T$ with ${\rm det}(B) \not= 0$.
There exists a change of coordinates ${\bf x} \mapsto {\bf x}g$ which transforms them to 
${\bf x} C {\bf x}^T$ and ${\bf x} D {\bf x}^T$ if and only if the matrices $A-\lambda B $ and $C - \lambda D$ have the same elementary divisors. 
\end{corollary}

\begin{proof}
For a textbook proof of this classical fact see \cite[Theorem~1, p.~278]{HP}.
\end{proof}

Corollary \ref{transCoord} is used to construct a canonical form for pencils.
For $e \in \NN$ and $\alpha \in \CC$,  we define a pair of symmetric $e \times e$ matrices 
by filling their antidiagonals:
\begin{equation}
\label{eq:blocks}
\begin{small}
P_e(\alpha) \,\,= \,\,
\begin{pmatrix}
0 & 0 & \cdots & 0 & \alpha\\ 
0 & 0 & \cdots & \alpha & 1\\
\vdots & \vdots & \udots & \udots & \vdots\\
0 & \alpha & 1 & \vdots & 0\\
\alpha & 1 & \cdots & 0 & 0
\end{pmatrix}
\quad {\rm and} \quad
Q_e \,\,= \,\,
\begin{pmatrix}
0 & \cdots & 0 & 0 & 1 \\
0 & \cdots & 0 & 1 & 0 \\
0 & \cdots & 1 & 0 & 0 \\
\vdots & \udots & \vdots & \vdots & \vdots \\
1 & \cdots &  0 & 0 & 0
\end{pmatrix}.
\end{small}
\end{equation}
The $e \times e$ matrix $\,P_e(\alpha) - \lambda Q_e\,$ has only one
elementary divisor, namely $(\lambda- \alpha)^e$.

Let us now start with the list in (\ref{eq:elemdiv}).
For each elementary divisor $(\lambda-\alpha_i)^{e_i}$ we form the $e_i \times e_i$ matrices in
(\ref{eq:blocks}), and we aggregate these blocks as follows:
\begin{equation}
\label{eq:normalform}
\begin{small}
P \,\,= \,\,
\begin{pmatrix}
\,P_{e_1}(\alpha_{1}) \! \! & 0 & \cdots & 0\\
\,0 & \! \! P_{e_2}(\alpha_{2}) \! \! & \cdots & 0\\
\,\vdots & \vdots & \ddots & \vdots\\
\,0 & 0 & \cdots & \! P_{e_s}(\alpha_{s})\,
\end{pmatrix}
\quad {\rm and} \quad
Q \,\,= \,\,
\begin{pmatrix}
\,Q_{e_1}& 0 & \cdots & 0 \, \\
\,0 & Q_{e_2} & \cdots & 0 \,\\
\,\vdots & \vdots & \ddots & \vdots \, \\
\,0 & 0 & \cdots & Q_{e_s}\,
\end{pmatrix}.
\end{small}
\end{equation}
The matrices $A-\lambda B$ and $P - \lambda Q$ have the same elementary divisors. 
Hence, by Corollary \ref{transCoord}, the pair $({\bf x} A {\bf x}^T,{\bf x} B {\bf x}^T )$ 
is isomorphic to  $({\bf x} P {\bf x}^T, {\bf x} Q {\bf x}^T)$ under the action by 
${\rm GL}(n)$. As in Example \ref{ex:n3}, every
regular pencil $\mathcal{L} \in {\rm Gr}(2,\SSS^n)$ has a 
normal form $\CC \{ P,Q\}$, where the matrices $P$ and $Q$  are
defined by the unordered list (\ref{eq:elemdiv}).
Given any Segre symbol $\sigma$, its canonical representative is
$\mathcal{L} = \CC \{P,Q\}$ where $\alpha_1,\ldots,\alpha_r$ are parameters.
 In what follows,  we often use index-free notation for unknowns, like
  ${\bf x} = (x,y,z)$ and $(\alpha_1,\alpha_2,\alpha_3) = (a,b,c)$.

\begin{example}[$n=5$]  \label{ex:212} Let  $\sigma = [(2,1),2]$. The 
list of elementary divisors~equals
$$ (\lambda-a)^2, \,(\lambda-a), \,(\lambda-b)^2. $$
Our canonical representative (\ref{eq:normalform})  for this
class of pencils $\mathcal{L}$ is the matrix pair
$$ \begin{small}
P \,\,= \,\,
\begin{pmatrix}
0 & a & 0 & 0 & 0\\
a & 1 & 0 & 0 & 0\\
0 & 0 & a & 0 & 0\\
0 & 0 & 0 & 0 & b\\
0 & 0 & 0 & b & 1
\end{pmatrix}
\quad  {\rm and} \quad
Q \,\,= \,\,
\begin{pmatrix}
0 & 1 & 0 & 0 & 0\\
1 & 0 & 0 & 0 & 0\\
0 & 0 & 1 & 0 & 0\\
0 & 0 & 0 & 0 & 1\\
0 & 0 & 0 & 1 & 0
\end{pmatrix}.
\end{small}
$$
The quadrics
$\,P = 2a xy + y^2 + az^2 + 2buv + v^2\,$ and $\,Q = 2xy + z^2 + 2 uv \,$
define a degenerate del Pezzo surface of degree four in $\PP^4$.
This surface has two singular points, $(0:0:0:1:0)$ and $(1:0:0:0:0)$; their multiplicities
are one and~three.
\end{example}

\begin{remark}
To appreciate Theorem~\ref{thm:WS}  and Corollary \ref{transCoord},
it helps to distinguish the two geometric figures associated with a pencil of quadrics,
and how the groups ${\rm GL}(2)$ and ${\rm GL}(n)$ act on these. First, there
is the configuration of $n$ points in $\PP^1$ defined by ${\rm det}(\mathcal{L})$.
This configuration undergoes projective transformations via ${\rm GL}(2)$ but it 
is left invariant by ${\rm GL}(n)$. Second, there is the codimension $2$ variety
in $\PP^{n-1}$ defined by the intersection of the two quadrics in $\mathcal{L}$. This variety
undergoes projective transformations via ${\rm GL}(n)$ but it is left invariant
by ${\rm GL}(2)$.  Hence, combining Theorem~\ref{thm:WS}  and Corollary \ref{transCoord}, we want these two geometric figures to be invariant when looking at isomorphic pencils, and this is possible by acting 
 on pencils with the two groups GL$(2)$ and GL$(n)$.
\end{remark}

In this section, pencils $\mathcal{L} = \CC\{A,B\}$ are studied by
linear algebra over a PID. We use
 the relationship between elementary divisors
and invariant factors.
One can compute these with the {\em Smith normal form} algorithm
over $\CC[\lambda]$.
We apply this to a specific torsion module,
namely the cokernel of our  matrix $A -\lambda B$.

Fix $n$ and a Segre symbol $\sigma = [ \sigma_1,\ldots,\sigma_r]$, where each entry is now
 a weakly decreasing vector $\sigma_i = (\sigma_{i1}, \sigma_{i2}, \ldots, \sigma_{in})$
of nonnegative integers. With this convention, the Segre symbol
$\sigma = [\sigma_1,\sigma_2]$  in Example \ref{ex:212}, with $n=5,s=3,r=2$, has 
$\sigma_1 = (2,1,0)$ and $\sigma_2 = (2,0,0)$.
Write $\alpha_1,\ldots,\alpha_r \in \CC$ for the distinct roots of ${\rm det}(A -\lambda B)$.
Then the elementary divisors are $(\lambda - \alpha_i)^{\sigma_{ij}}$ for $i=1,\ldots,r$ and $j=1,\ldots,n$.
Only $s$ of these  are different from $1$. The invariant factors are
$$ d_j \,\,\, := \,\,\, \prod_{i=1}^r (\lambda-\alpha_i)^{\sigma_{ij}} \quad {\rm  for}  \,\,\, j=1,\ldots,n. $$
Note that $\,d_n \,|\, d_{n-1}\, |\, \cdots \,|\, d_2 \,|\, d_1$.
The number of nontrivial invariant factors is the maximum number of
parts among the $r$ partitions $\sigma_i$.
For instance, in Example \ref{ex:212},  the invariant factors are
$\,d_1 = (\lambda - a)^2 (\lambda -b )^2 , \,d_2 = \lambda -a , \,d_3 = d_4 = d_5 = 1$.

The ideal of $k \times k$ minors of 
$A - \lambda B$ is generated by the greatest common divisor $D_k$ of these minors. 
The theory of modules over a PID tells us that
\begin{equation}
\label{eq:minors} D_k \quad := \quad
 \prod_{j=1}^{k} d_{n+1-j}  \quad = \quad \prod_{i=1}^r ( \lambda - \alpha_i)^{\sigma_{i,n-k+1} + \cdots + 
\sigma_{i,n-1} + \sigma_{i,n}}. 
\end{equation}

The Segre symbol  of  a  pencil
$\mathcal{L} = \CC \{A,B\}$ is determined by the
 ideal of $k \times k$ minors of $A-\lambda B$ for $k=1,\ldots,n$.
 In practice, we use the Smith normal form of $A-\lambda B$.
In the Introduction we proposed a different method, namely 
 the Jordan canonical form of $AB^{-1}$.
This computation uses only linear algebra over $\CC$,
unlike the Smith normal form.
To see that the Jordan canonical form of $AB^{-1}$ reveals
the Segre symbol, consider the transformation from $(A,B)$
to $(P,Q)$ in  Corollary \ref{transCoord}. This 
preserves the conjugacy class of $AB^{-1}$. Therefore,
$AB^{-1}$ and $PQ^{-1}$ have the same Jordan canonical form.
We  see in (\ref{eq:normalform}) that $Q$ is a permutation matrix,
and hence so is $Q^{-1}$. Furthermore, $P$ is already in Jordan
canonical form, after permuting rows and columns, and $\sigma$ is clearly visible in~$P$.

\section{The Reciprocal Curve}
\label{sec3}

For any regular pencil $\mathcal{L}$, we are interested in
the reciprocal curve $\PP \mathcal{L}^{-1}$.
We write ${\rm deg}(\mathcal{L}^{-1})$ for the degree of this curve in
$\PP(\SSS^n) $.
In Example \ref{ex:n3}, we have ${\rm deg}(\mathcal{L}^{-1}) = 2$ in  three cases, so
$\PP \mathcal{L}^{-1}$ is a plane conic. In the other two cases,
$\PP \mathcal{L}^{-1}$ is a line in $\PP^5$.
Here are the homogeneous prime ideals of these~curves:
$$
\begin{small}
\begin{matrix} 
\hbox{Segre symbol}  &  \hbox{Ideal of the reciprocal curve $\PP \mathcal{L}^{-1} $} & {\rm mingens} \\
[1,1,1] & \langle   x_{12}, x_{13}, x_{23}, (c{-}b) x_{11} x_{22} + (a{-}c) x_{11} x_{33} + (b{-}a) x_{22} x_{33} \rangle  
& (3,1)  \\
[2,1] & \langle  \,x_{13}\,, \,x_{22}\,, \,x_{23}\,,\, x_{12}^2 + (c-a) x_{11} x_{33} - 2 x_{12} x_{33}  \,\rangle & (3,1)  \\
[\,3\,]  & \langle \,  x_{23}\,, \,x_{33}\,, \,x_{13} - 2 x_{22}\,,\, x_{12}^2 - x_{11} x_{22}  \,\rangle & (3,1)   \\
[(1,1),1] &  \langle  \, x_{12}\,, \,x_{13}\,,\, x_{23}\,, \,x_{11} - x_{22}\,  \rangle & (4,0)   \\
[(2,1)] & \langle  \, x_{13}\,,\, x_{22}\,, \, x_{23}\,,\, x_{12} - 2 x_{33}  \,\rangle & (4,0)  \\
\end{matrix}
\end{small}
 $$
 The column ``mingens'' gives the numbers of linear and quadratic generators.

\begin{example}[$n=4$] \label{ex:n4}
Two quadrics $P$ and $Q$ in $\PP^3$
 meet in a quartic curve.
There are $13$ cases, one for each Segre symbol.
Here, $x,y,z,u$ are coordinates on~$\PP^3$.
$$
\begin{small}
\begin{matrix} 
\hbox{symbol} & \!\!\!\!\! \!\! {\rm codims} \!\!\!\!\! &\!  {\rm degrees}& \!\! \!{\rm mingens}\!\!\! \!& 
\!\! {\rm quadrics} \,\, P , Q
 & \hbox{variety in $\PP^3$} 
\medskip \\
[1,1,1,1]  & 0,0,0 & ({\bf 3},3,5) & (6,3) & {ax^2+by^2+cz^2+du^2  \atop x^2+y^2+z^2+u^2 } & \hbox{\em elliptic curve} \smallskip \\  
[2,1,1] & 1,1, 1 & ({\bf 3},2,4)  &  (6,3) & {2axy+y^2+cz^2+du^2 \atop 2xy+z^2 +u^2} & \hbox{\em nodal curve} \smallskip  \\
[(1,\!1),1,1] & 3,2,2 & ({\bf 2},2,3) & (7,1) & {a(x^2+y^2)+cz^2+du^2 \atop x^2+y^2+z^2+u^2 } &
 \hbox{\em two conics meet twice} \smallskip \\
[3,1] & 2,2,2 & ({\bf 3},1,3) &  (6,3) & {2axz+ay^2+2yz+du^2 \atop 2xz+y^2+u^2}& \hbox{\em cuspidal curve} \smallskip  \\
[2,2] & 2,2, 2 &  ({\bf 3},1,3) & (6,3) &  { 2axy+y^2+2bzu+u^2 \atop 2xy+2zu}& \! \hbox{\em twisted cubic with secant} \smallskip \\
[(2,1),1] & 4,3,3 & ({\bf 2},1,2)  &  (7,1) & { 2axy+y^2+az^2+du^2 \atop 2xy+z^2+u^2} & \hbox{\em two tangent conics} \smallskip\\
[4] &3,3, 3 & ({\bf 3},0,2) &  (6,3) & { 2axu+2ayz+2yu+z^2 \atop 2xu+2yz }& \!\!\! \hbox{\em twisted cubic with tangent} \smallskip \\
[2,(1,1)] & 4,3,3 & ({\bf 2},1,2) & (7,1) & { 2axy+y^2+c(z^2+u^2) \atop 2xy+z^2+u^2} & \hbox{\em conic meets two lines} \smallskip \\
[(3,1)] & 5,4,4 & ({\bf 2},0,1)  &  (7,1) & { 2axz+ay^2+2yz+au^2 \atop 2xz+y^2+u^2}& \!\! \! \hbox{\em conic and two lines concur} \smallskip \\
\!\! [(1,\! 1),(1,\! 1)] \!\!&6,4, 4 &({\bf 1},1,1) &  (8,0) & { a(x^2+y^2)+c(z^2+u^2) \atop x^2+y^2+z^2+u^2 } & \hbox{\em quadrangle of lines} \smallskip  \\
[(1,1,1),1] & 8,5,5 & ({\bf 1},1,1) &  (8,0) & { a(x^2+y^2+z^2)+du^2 \atop x^2+y^2+z^2+u^2}& \hbox{\em double conic} \smallskip  \\
[(2,2)] & 7,5,5 & ({\bf 1},0,0) & (8,0) & { 2axy+y^2+2azu+u^2 \atop 2xy+2zu} & \hbox{\em double line and two lines} \smallskip\\
[(2,1,1)] & 9,6,6 & ({\bf 1},0,0)  &  (8,0) & { 2axy+y^2+a(z^2+u^2) \atop 2xy+z^2+u^2} & \hbox{\em two double lines} \smallskip \\
 \end{matrix}
 \end{small}
$$
We see that $\PP \mathcal{L}^{-1} \subset \PP^9$ is either
a line, a plane conic, or a twisted cubic~curve.
 This is explained by the next theorem, which is our main result in Section~\ref{sec3}.
\end{example}

\begin{theorem} \label{thm:eight}
Let $\mathcal{L}$ be a regular pencil in $\SSS^n$ with Segre
symbol $\sigma = [\sigma_1,\ldots, \sigma_r]$.
Then $\PP \mathcal{L}^{-1}$ 
 is a rational normal curve of degree $d$ in $\PP (\SSS^n)$,
 where $d  = \sum_{i=1}^r \sigma_{i1} -1$ is one less than the
 sum of the first parts of the partitions in $\sigma $.
 The ideal of $\PP \mathcal{L}^{-1}$
 is generated by $\binom{n+1}{2} -  d - 1$ linear forms
 and $\binom{d}{2}$ quadrics in $\binom{n+1}{2}$ unknowns.
\end{theorem}

\begin{proof}
The curve $\PP \mathcal{L}^{-1}$ is parametrized by  $\binom{n+1}{2}$
rational functions in one unknown $\lambda$, namely the
entries in the inverse of matrix $P - \lambda Q$ in Section~\ref{sec2}. 
We scale each entry by $D_n = \pm{\rm det}(P - \lambda Q)$ to get a
polynomial parametrization by the adjoint of $P- \lambda Q$. This is
an $ n \times n$ matrix 
whose entries are the $(n{-}1) \times (n{-}1)$ minors of $P-\lambda Q$. These
are polynomials of degree $\leq n-1$ in~$\lambda$, which are
divisible by the invariant factor $D_{n-1}$. Note that $D_{n-1}$ has
degree $\sum_{i=1}^r \sum_{j=2}^n \sigma_{ij}$ in $\lambda$.
 Subtracting this
 from the expected degree $n-1$,  we obtain 
 $d  = \sum_{i=1}^r \sigma_{i1} -1$.
 We remove the factor $D_{n-1}$ from each entry of
 the adjoint. The resulting matrix $(D_n/D_{n-1})\cdot (P-\lambda Q)^{-1}$ 
 also parametrizes $\PP \mathcal{L}^{-1}$.
 The entries of that matrix  are polynomials in $\lambda$ of degree $\leq d$.
 As a key step, we will show that these span
 the $(d+1)$-dimensional space $\CC[\lambda]_{\leq d}$ of all polynomials in $\lambda$ 
  of degree $\leq d$.


The inverse of $P-\lambda Q$ is a block matrix, where the blocks are
the inverses~of the $e \times e$ matrices $P_e(\alpha) - \lambda Q_e$ in (\ref{eq:blocks}),
one for each elementary divisor. A  computation shows that  the entry of 
$(P_e(\alpha) - \lambda Q_e)^{-1}$ in row $i$ and column $j$~is
\begin{equation}
\label{eq:inverseentries}
  -(\lambda - \alpha )^{i+j-e-2}  \quad {\rm if}\,\,\, i+j \leq e+1 
\qquad {\rm and} \qquad
0  \quad {\rm  if}\,\,\, i+j \geq e+2.
\end{equation}
It follows that the distinct nonzero entries in the $n \times n$ matrix
$(P-\lambda Q)^{-1}$ are
\begin{equation}
\label{eq:wantLI} \pm (\lambda - \alpha_i)^{-k}  \quad
\hbox{where $\,1 \leq k \leq \sigma_{i1}\,$ and $\,1 \leq i \leq r$.}
\end{equation}
The common denominator of these $d+1 = \sum_{i=1}^r \sigma_{i1}$
rational functions in $\lambda$ is equal to
$\,D_n/D_{n-1} = \prod_{i=1}(\lambda - \alpha_i)^{\sigma_{i1}} $.
Multiplying  by that common denominator, we
obtain $d+1$  polynomials in $\lambda$ of degree $\leq d$.
Lemma~\ref{lem:negativepowers}  below tells us
that these polynomials are linearly independent. Hence they span 
 $\CC[\lambda]_{\leq d} \simeq \CC^{d+1}$.

The proof  of Theorem \ref{thm:eight} now concludes as follows. 
By recording which entries of $(P-\lambda Q)^{-1}$ are zero,
and which pairs of entries are equal, we obtain
$\binom{n+1}{2} -d-1$  independent linear forms that vanish on 
$\PP \mathcal{L}^{-1}$. We know that there exist
linear forms $u_i$ in the matrix entries which
evaluate to $\lambda^i$ for $i=0,1,2,\ldots,d$.
The $\binom{d}{2}$ quadrics that vanish on $\PP \mathcal{L}^{-1}$
are the $2 \times 2$ minors~of the $2 \times d$ matrix
\begin{equation} \label{eq:2byd} \begin{small}
\begin{pmatrix} 
u_0 & u_1 & u_2 & \cdots & u_{d-1} \\
u_1 & u_2 & u_3 & \cdots & u_d \end{pmatrix} . \end{small}
\end{equation}
We have thus constructed an  isomorphism between our curve
 $\PP \mathcal{L}^{-1} $ and the rational normal curve $
\{ ( 1 : \lambda :  \cdots : \lambda^d) \}$,
whose prime ideal is given by (\ref{eq:2byd}).
\end{proof}
Notice that the final part of the proof gives an algorithm for computing generators
of the homogeneous prime ideal that defines the reciprocal curve.

\begin{lemm} \label{lem:negativepowers}
A finite set of distinct rational functions $(\lambda - \alpha_j)^{-s_{ij}}$,
each a negative power of one of the expressions
$\lambda- \alpha_1,\ldots, \lambda - \alpha_r$, is linearly independent.
\end{lemm}

\begin{proof}
We use induction on  $r$. The base case is $r=1$.
We claim that $(\lambda-\alpha)^{-s_1}, \,\ldots ,$ $ (\lambda-\alpha)^{-s_n}$  are 
linearly independent when $0<s_1< \cdots <s_n$. Suppose 
$$ k_1(\lambda-\alpha)^{-s_1}+ \,\cdots \,+k_n(\lambda-\alpha)^{-s_n}\,\,=\,\,0
\qquad \hbox{for some} \,\,\, k_1, \ldots ,k_n \in \mathbb{C}. $$
Clearing denominators, we obtain
$k_1(\lambda-\alpha)^{s_n-s_1}+ \cdots +k_n=0$.
Setting $\lambda=\alpha$ we find $k_n=0$. Repeating this computation $n$ times,
we conclude  $k_1=k_2= \cdots =k_n=0$.

For the induction step from $r-1$ to $r$, we 
consider  distinct negative powers
\begin{equation}
\label{eq:rrows}
\begin{matrix}
(\lambda-\alpha_1)^{-s_{1,1}}, & (\lambda-\alpha_1)^{-s_{1,2}}, & \!\! \dots \,\, , & (\lambda-\alpha_1)^{-s_{1,n_1}}, \\
\vdots & \vdots & & \vdots \\
(\lambda-\alpha_r)^{-s_{r,1}}, & (\lambda-\alpha_r)^{-s_{r,2}}, & \!\! \dots \,\, , & (\lambda-\alpha_r)^{-s_{r,n_r}},
\end{matrix}
\end{equation}
where $0\leq s_{i,j} <s_{i,j+1}$ for  $i=1,...,r$ and $j=1,...,n_i$.
 Consider a linear combination of (\ref{eq:rrows}) with coefficients
 $k_{1,1},\ldots,k_{r,n_r} $.
  Multiplying by $(\lambda- \alpha_r)^{s_{r,n_r}}$ and setting $\lambda = \alpha_r$,
we find $k_{r,n_r} = 0$.  Repeating with 
$(\lambda- \alpha_r)^{s_{r,i}}$ for $i=n_r {-}1,n_r {-}2,\ldots,1$, we get
$k_{r,1} = \cdots = k_{r,n_r} = 0$. By the induction hypothesis, 
the first $r-1$ rows of (\ref{eq:rrows}) are linearly independent.
This proves that all $k_{i,j}$ are zero. Lemma \ref{lem:negativepowers} follows.
\end{proof}

The last paragraph in the proof of Theorem  \ref{thm:eight}
gives an algorithm for computing generators of the ideal of
$\PP \mathcal{L}^{-1}$. We show this for our running example.

\begin{example} \label{ex:212ideal}
Let $\sigma = [(2,1),2]$ as in Example \ref{ex:212}.
We have $d = \sigma_{11} + \sigma_{21} -1 = 3$,
so $\PP \mathcal{L}^{-1}$ is a twisted cubic curve in $\PP^{14}$.
The inverse of $P- \lambda Q$ satisfies the $\binom{6}{2} -3-1 = 11$ linear forms
$\,x_{13}, x_{14}, x_{15}, x_{22}, x_{23}, x_{24}, x_{25}, x_{34}, x_{35}, x_{55} ,\,
x_{12} - x_{33}$. The quadratic ideal generators are 
$u_0 u_2 - u_1^2$, $u_0 u_3 - u_1 u_2$ and $u_1 u_3- u_2^2$, where
 $$  \begin{small} \begin{matrix}
u_0 &= &  (a-b)  x_{11}\,-\,2  x_{12}\,+\,(a-b)  x_{44}\,+\,2  x_{45}\,, \\
u_1 & = &  (a^2-a b)  x_{11}-(a+b)  x_{12}+(a b-b^2)  x_{44}+(a+b)  x_{45}\,,  \\
u_2 & = &  (a^3-a^2 b)  x_{11}\, -\, 2 a b  x_{12}\,+\,(a b^2-b^3)  x_{44} \,+\,2 a b  x_{45}\,, \\
u_3 & = &  (a^4-a^3 b)  x_{11}+(a^3-3 a^2 b)  x_{12} + (a b^3-b^4)  x_{44} + (3 a b^2-b^3)  x_{45}.
\end{matrix} \end{small}
$$
Note that
$\,x_{11} = -(\lambda - a)^{-2} $,
$x_{12} = (\lambda-a)^{-1} $,
$x_{44} = -(\lambda-b)^{-2} $,
$x_{45} = (\lambda-b)^{-1} $.
\end{example}
 
\section{Maximum Likelihood Degrees}
\label{sec4}

Let $ \SSS^n_{\succ 0}$ denote the open convex cone of
positive definite real symmetric $n \times n$ matrices.
For any fixed $S \in \SSS^n$, 
we consider the following {\em log-likelihood function}:
\begin{equation}
\label{eq:loglikelihood}
 \ell_S \,:\,\SSS^n_{\succ 0}\, \rightarrow\, \RR \, , \,\,
M \,\mapsto \,{\rm log}({\rm det}(M)) - {\rm trace}(S M). 
\end{equation}
We seek to compute the critical points of $\ell_S$ restricted
to a smooth subvariety of $\SSS^n$. 
Here, by a {\em critical point} we mean a nonsingular matrix $M$
in the subvariety whose normal space contains the gradient vector of $\ell_S$ at $M$.
This is an algebraic problem because
the $\binom{n+1}{2}$ partial derivatives of $\ell_S$ are rational functions. 

The determinant and the trace of a square matrix are invariant
under conjugation. This implies the following identity for all
invertible $n \times n$ matrices~$g$:
\begin{equation}
\label{eq:invariance}
 \ell_{g^{-1} S (g^{-1})^T} (g^T M g) \,\,= \,\,
{\rm log}({\rm det}(g^T M g)) - {\rm trace}(g^{-1} S   M g)
\,\,=\,\, \ell_S(M) + {\rm const}. 
\end{equation}

Let $\mathcal{L}$ be a linear subspace  of $\SSS^n$,
and fix a generic matrix $S \in \SSS^n$.
The {\em ML degree} ${\rm mld}(\mathcal{L})$
   is the number of complex critical points
of $\ell_S$ on $\mathcal{L}$.
The {\em reciprocal ML degree} ${\rm rmld}(\mathcal{L})$
 of $\mathcal{L}$ is the number of complex critical points
of $\ell_S$ on $\mathcal{L}^{-1}$.
Both ML degrees do not depend on the choice of $S$, as long as
$S$ is generic.
The ML degrees are invariant 
under the action of ${\rm GL}(n)$  by congruence on $\SSS^n$:

\begin{lemm} \label{lem:invariance}
The ML degree and  the reciprocal ML degree of 
a subspace $\mathcal{L} \subset \SSS^n$ 
are determined by its congruence class. In particular, this holds for two-dimensional subspaces
$\mathcal{L}$, i.e.~for pencils of quadrics.
\end{lemm}

\begin{proof} Fix $g$ and $\mathcal{L}$.
If the matrix $S$ is  generic in $\SSS^n$ then so is  $g^{-1} S (g^{-1})^T$.
The image of $\mathcal{L}$ under congruence by $g^T$ consists of all
matrices $g^T M g$ where $M \in \mathcal{L}$.
By (\ref{eq:invariance}),
 the likelihood function of $S$ on $\mathcal{L}$ agrees
with that of $g^{-1} S (g^{-1})^T$ on $g^T \mathcal{L} g$, up to an additive constant.
The two functions have the same number of critical points, so
the subspaces $\mathcal{L}$ and $g^T \mathcal{L} g$ have the same ML degree.
The same argument works if $\mathcal{L}$ is replaced by any
nonlinear variety, such as  $\mathcal{L}^{-1}$.
\end{proof}

We now focus on  pencils ($m=2$), and we state
our main result in Section~\ref{sec4}.

\begin{theorem} \label{thm:sec4main}
Let $\mathcal{L}$ be a pencil with
Segre symbol $\sigma = [\sigma_1,\ldots,\sigma_r]$. Then
\begin{equation} \label{eq:mlformulas}
\!\!  {\rm mld}(\mathcal{L}) \,=\, r-1 \,\,\,\,\, {\rm and} \,\,\,\,\,
{\rm rmld} (\mathcal{L}) \,=\, \sum_{i=1}^r \sigma_{i1} + r - 3\,\,=\,\,
{\rm deg}(\mathcal{L}^{-1}) + {\rm mld}(\mathcal{L}) - 1 . 
\end{equation}
\end{theorem}

For generic subspaces $\mathcal{L}$, with
Segre symbol $\sigma = [ 1, \ldots, 1]$, this implies
\begin{equation}
\label{eq:genericformulas}
 {\rm mld}(\mathcal{L}) \,=\, 
 {\rm deg}(\mathcal{L}^{-1}) \,=\,n-1 \qquad {\rm and} \qquad {\rm rmld}(\mathcal{L})\, =\, 2n-3. 
\end{equation}
The left formula in (\ref{eq:genericformulas}) appears in \cite[Section 2.2]{SU}. The
right formula in (\ref{eq:genericformulas}) is due to  Coons, Marigliano and Ruddy \cite{CMR}.
We here generalize these results to arbitrary
  pencils~$\mathcal{L}$. The proof  of 
  Theorem \ref{thm:sec4main} 
    appears at the end of this section.

\smallskip

The log-likelihood function (\ref{eq:loglikelihood}) is important
in statistics. 
The sample covariance matrix $S$
encodes data points in $\RR^n$. The matrix $M$ is the
concentration matrix. Its inverse $M^{-1}$ is the covariance matrix.
These represent Gaussian distributions on $\RR^n$.
The subspace $ \mathcal{L}$ encodes linear constraints, either on $ M$ or on $M^{-1}$.
For the former, we get the ML degree. For the latter, we get the reciprocal ML degree.
These degrees measure the algebraic complexity of  maximum likelihood estimation.
In the language in \cite{CMR,STZ},
${\rm mld}(\mathcal{L})$ refers to the
{\em linear concentration model,} while ${\rm rmld}(\mathcal{L})$ refers to the
{\em linear covariance~model}.

If $\mathcal{L}$ is a statistical model,
then it contains a positive definite matrix. In symbols,
$\mathcal{L} \cap \SSS^n_{\succ 0} \not= \emptyset$.
If  this holds and ${\rm dim}(\mathcal{L}) = 2$
then $\mathcal{L}$ is called a {\em $d$-pencil} \cite{uhlig79}.
Thus, our numbers
${\rm mld}(\mathcal{L})$ and ${\rm rmld}(\mathcal{L})$
are interesting for statistics when $\mathcal{L}$ is a
$d$-pencil.
Here, we can take advantage of the following  linear algebra fact.

\begin{lemm} \label{lem:a1a2}
Every $d$-pencil $\mathcal{L}$ can be simultaneously diagonalized over $\RR$.
After a  change of coordinates,  $\mathcal{L}$ is spanned by the  quadrics $\sum_{i=1}^n a_i x_i^2$ and
$ \sum_{i=1}^n x_i^2$.
\end{lemm}

\begin{proof}
We assume $n \geq 3$. 
A pencil is a $d$-pencil if and only if it has no zeros in the
real projective space $\PP^{n-1}$. This is the
Main Theorem in \cite{uhlig79}. It was also proved
by Calabi in \cite{calabi}. The fact that pencils without real zeros in $\PP^{n-1}$
can be diagonalized is \cite[page 221, (PM)]{uhlig79}.
It is also Remark 2 in \cite[page 846]{calabi}.
\end{proof}

Suppose there are $r$ distinct elements in $\{a_1,\ldots,a_n\}$.
Theorem \ref{thm:sec4main} implies:

\begin{corollary} \label{cor:mlddpen}
If $\mathcal{L}$ is a $d$-pencil then
 ${\rm mld}(\mathcal{L}) = {\rm deg}(\mathcal{L}^{-1}) = r-1$ and
 ${\rm rmld}(\mathcal{L}) = 2r-3$,
 where $\mathcal{L}$ has $r$ distinct eigenvalues.
 This holds for all subspaces $\mathcal{L}$ that represent statistical models,
 since such an $\mathcal{L}$ contains positive definite matrices.
\end{corollary}

The log-likelihood  function for our $d$-pencil $\mathcal{L}$ can be written as follows:
$$ \ell_S(x,y) \,\,=\,\, \sum_{i=1}^n\, \bigl( \,{\rm log}(a_i x + y) - s_i (a_i x + y) \,\bigr). $$
Here $s_1,\ldots,s_n  \in \RR$ represent data.
The MLE is the maximizer of $\ell_S(x,y)$ over the 
cone $\{(x,y) \in \RR^2 :a_i x + y > 0 \,\,{\rm for} \,\,i=1,\ldots,n\}$.
Corollary \ref{cor:mlddpen} says that $\ell_S(x,y)$ has $r-1$ critical points. One of them
is the MLE.
The reciprocal log-likelihood~is
\begin{equation}
\label{eq:recell} \begin{small}
\tilde \ell_S(x,y) \,\,=\,\, \sum_{i=1}^n \,\Bigl( \,-{\rm log}(a_i x + y) \,-\, \frac{s_i}{a_i x + y} \,\Bigr).  
\end{small}
 \end{equation}
The invariant ${\rm rmld}(\mathcal{L})$ is the number of critical points 
$(x^*,y^*)$ of this function with $\prod_{i=1}^n(a_i x^* + y^*) \not= 0$,
provided $s = (s_1,\ldots,s_n)$ is generic in $\RR^n$. 
Corollary \ref{cor:mlddpen} states that $\tilde \ell_S(x,y)$ has $2r-3$ complex critical points. One of them is the~MLE.

The following is an extension of a conjecture stated by
Coons et al.~\cite[\S 6]{CMR}.

\begin{conjecture}
Let $\mathcal{L}$ be a $d$-pencil with $r$ distinct eigenvalues.
There exists $s = (s_1,\ldots,s_n) \in \RR^n$ such that 
the function  (\ref{eq:recell}) has $2r-3$ distinct real critical~points.
\end{conjecture}

We can prove this conjecture for small values of $n$ by explicit computation.

\begin{example} \rm
Fix the pencil $\mathcal{L}$ with $n=r$ and
$(a_1,\ldots,a_n) = (1,\ldots,n)$. For $n \leq 7$ we found
$s \in \RR^n$ such that
the reciprocal log-likelihood function $\tilde \ell_s$ has $2n-3$ distinct real critical~points.
For $n=7$ we can take
$s = (-\frac{74}{39}, \frac{13}{47}, \frac{61}{40}, \frac{1}{7},\frac{23}{18},-73,-\frac{27}{43})  $.
\end{example}

We now return to arbitrary Segre symbols $\sigma$.
While non-diagonalizable pencils $\mathcal{L}$ do not arise in
applied statistics,  their likelihood geometry is interesting.

\begin{proof}[Proof of Theorem \ref{thm:sec4main}]
By Lemma \ref{lem:invariance}, we may assume that $\mathcal{L}$
is parametrized by $(x,y) \mapsto x P - y Q$ with $P$ and $Q$ as in
 (\ref{eq:normalform}). For generic $S \in \SSS^n$, we seek the number ${\rm mld}(\mathcal{L})$
 of critical points in $\CC^2$ of the following  function in two variables:
 \begin{equation}
 \label{eq:loglikePQ}
  \ell_S(x,y)\, \,\,= \,{\rm log}({\rm det}(xP-yQ))  \, - \, {\rm trace}( S (xP-yQ)).   
  \end{equation}
  After multiplying by  $d = \prod_{i=1}^r (\alpha_i x - y)$,
 the two partial derivatives of $\ell_S(x,y)$ have the form
$ f(x,y)\,= \,\lambda_S d\,+\,C\,$ and $\, g(x,y) \,=\,\mu_S d\,+\,D$.
Here $\lambda_S = - {\rm trace}(SP)$~and $\mu_S = {\rm trace}(SQ)$ are constants, 
and the following are binary forms of degree~$r-1$:
\begin{equation}
\label{eq:CandD}
  C \,\,= \,\,\sum_{i=1}^r \sum_{j=1}^n \sigma_{ij} \,\alpha_i \prod_{k=1,\\ k\neq i}^r (\alpha_k x - y) \quad {\rm and} \quad D \,\,= \,\, - \sum_{i=1}^r \sum_{j=1}^n \sigma_{ij} \prod_{k=1,\\ k\neq i}^r (\alpha_k x - y).
  \end{equation}

The variety of critical points of $\ell_{S}$ in $\CC^2$  is $ V(f,g) \backslash V(d) $.
We adapt the method introduced in \cite{CMR} to enumerate this set.
 Let $F(x,y,z)$ and $G(x,y,z)$ denote the homogenizations of $f$ and $g$ with respect to $z$. 
Both $F$ and $G$  define curves of degree $r$ in $\PP^2$. Since $F$ and $G$ do not share a common component, we can apply B\'ezout's Theorem to count their intersection points. This tells us that
\begin{equation}\label{eq:bezout} \begin{small}
{\rm mld}(\mathcal{L}) \,\,\,= \,\,\, r^2 \,-\,\,\, I_{[0:0:1]} (F,G) \,\,\,- \!\! \sum_{q \in V(F,G,z)} \!\!
I_q(F,G). \end{small}
\end{equation}
The negated expressions
are the intersection multiplicities of $F$ and $G$ at the origin
and on the line at  infinity. By computing these two quantities, we obtain
$${\rm mld}(\mathcal{L})\,\,=\,\,r^2 - (r-1)^2 -r \,\,=\,\, r-1. $$

The proof of the second formula in (\ref{eq:mlformulas})
is analogous but the details are more delicate. We present an outline.
The log-likelihood function for $\mathcal{L}^{-1}$ equals
\begin{equation*} \begin{small} 
\tilde{\ell}_S(x,y) \,\,=\,\, - \, \log\, \Bigl(\prod_{i=1}^r (\alpha_i x - y)^{\sigma_{i1}+\cdots+\sigma_{in}}\Bigr)
\,\, -\,\, \sum_{i=1}^r \sum_{j=1}^{\sigma_{i1}}\tilde{s}_{ij} \, \frac{x^{j-1}}{(\alpha_i x-y)^{j}},
\end{small}
\end{equation*}
where the $\tilde{s}_{ij}$ are linear combinations of the entries in the matrix $S$. This is obtained by
replacing the matrix $xP-yQ$  in  (\ref{eq:loglikePQ}) with its inverse. We find
\begin{equation}\label{eq:reciprocal derivatives}
\begin{small}
\begin{matrix}
\tilde{\ell}_{{S}_x} \,\,= \,\,-\,\displaystyle{\sum_{i=1}^r \sum_{j=1}^n} \frac{\sigma_{ij}\alpha_i}{\alpha_i x - y} 
\,\,+\,\,\sum_{i=1}^r \sum_{j=1}^{\sigma_{i1}} \tilde{s}_{ij} \frac{(j-1)x^{j-2}(\alpha_i x-y)-j \, x^{j-1}\alpha_i}{(\alpha_i x -y)^{j+1}}
\, , \medskip \\
\quad \tilde{\ell}_{{S}_y} \,\,= \, \quad \displaystyle{\sum_{i=1}^r} \sum_{j=1}^n \frac{\sigma_{ij}}{\alpha_i x - y} 
\,\,+\,\,\sum_{i=1}^r \sum_{j=1}^{\sigma_{i1}} \,\tilde{s}_{ij}\,\frac{j \,  x^{j-1}}{(\alpha_i x-y)^{j+1}} .\hspace*{1cm}
\end{matrix}
\end{small}
\end{equation}
We claim that the number of common zeros of 
the two partial derivatives $ \tilde{\ell}_{{S}_x}$ and $\tilde{\ell}_{{S}_y}$
in $\CC^2 \backslash V(d)$
is equal to $\, \varphi + r - 3\,$ where
$\, \varphi= \sum_{i=1}^r \sigma_{i1} = {\rm deg}(\mathcal{L}^{-1}) + 1$,

Clearing denominators in (\ref{eq:reciprocal derivatives}) yields polynomials
$-d' C + U$ and $-d' D + V$, where $d' = \prod_{i=1}^r (\alpha_i x - y)^{\sigma_{i1}}$,
the binary forms $U,V$ have degree $\varphi + r-2$, and
$C,D$ are precisely as in (\ref{eq:CandD}). Hence
${\rm deg}(d')= \varphi$ and ${\rm deg}(C) = {\rm deg}(D) = r-1$.
As before, these are sums of binary forms in consecutive degrees. We
use (\ref{eq:bezout}) to count their zeros in $\PP^2$.
We find
$(\varphi+r-1)^2 - (\varphi+r-2)^2 - (\varphi+r) = \varphi+r-3$
\end{proof}

\begin{example}[$n=5$]  Let $\sigma = [(2,1),2]$ as in Example \ref{ex:212}.
 The ML degrees are
${\rm mld}(\mathcal{L}) = 1$ and ${\rm rmld}(\mathcal{L}) = 3$.
Restricting the log-likelihood function to $\mathcal{L}$ gives
$$ \begin{small}  \ell_S \,=\, 
{\rm log}\bigl((ax-y)^3 (b x-y)^2 \bigr) \,
+\,2 s_{12} (a x-y)+s_{22} x+s_{33} (a x-y)+2 s_{45} (b x-y)+s_{55} x. \end{small}
$$
Its two partial derivatives are rational functions in $x$ and $y$.
Equating these to zero, we find that $\ell_S$ 
has  a unique critical point $(x^*,y^*)$ in $\mathcal{L}$. Its coordinates are
$$ \begin{small} \begin{matrix}
& \!\!\!\! x^* & = & \bigl(\, 4 ( a- b) s_{12}+5 s_{22}+ 2 ( a- b) s_{33}- 6 (b-a) s_{45}+5 s_{55} \,\bigr) \, /\, \Delta,
\qquad \qquad \qquad \qquad \qquad \\
&\!\! \!\! y^* & = &\,\, \bigl(\,
4a (a- b) s_{12}+(2 a+3 b) s_{22}+2a ( a- b) s_{33}+6b (b-a ) s_{45}+(2 a+3 b) s_{55}\,\bigr) /\, \Delta , \\
&
\!\! \Delta & = & \bigl(-s_{22}+2( a- b) s_{45}-s_{55} \bigr) \cdot \bigl(2( a- b) s_{12}+s_{22}+(a-b) s_{33}+s_{55} \bigr).
\qquad \qquad \quad \quad
\end{matrix} \end{small}
$$
The  restriction of the log-likelihood function to the reciprocal variety  $\mathcal{L}^{-1}$ is
$$  \! \tilde{\ell}_S(x,y) \, = \, \begin{small}
- {\rm log}\bigl((a x-y)^3 (b x-y)^2 \bigr) \,\,-\,  \frac{s_{11} \,x}{(a x-y)^2}+ \frac{2 \,s_{12}}{a x-y}
+ \frac{s_{33}}{a x-y} - \frac{s_{44} \,x}{(b x-y)^2}+ \frac{2\, s_{45}}{b x-y}.
\end{small} $$
The two partial derivatives have $3$ zeros, expressible in radicals
in $a,b,s_{11},\ldots,s_{45}$.
\end{example}

\section{Strata in the Grassmannian}
\label{sec5}

We now define a partial order on the set ${\rm Segre}_n$ of all Segre symbols for fixed $n$.
If $\sigma$ and $\tau$ are in  ${\rm Segre}_n$ 
then we say that $\sigma $ {\em is above} $ \tau$ if $|\sigma| > |\tau|$
and $\tau$ is obtained from $\sigma$ by replacing two partitions $\sigma_i, \sigma_j$
by their sum, or if $|\sigma| = |\tau|$ and
$\sigma$ and $\tau$ differ in precisely one partition, 
with index $i$, and $\tau_i \vartriangleleft \sigma_i$ in the dominance order on partitions.
The partial order on ${\rm Segre}_n$ is the transitive closure of the relation ``is above''.
The top element of our poset is $ [1 ,1 ,\ldots, 1]$,
and the bottom element is $[(2 ,1 ,\ldots, 1)]$.
The Hasse diagrams for $n=3,4$ are shown in Figure~\ref{fig:segre34}.

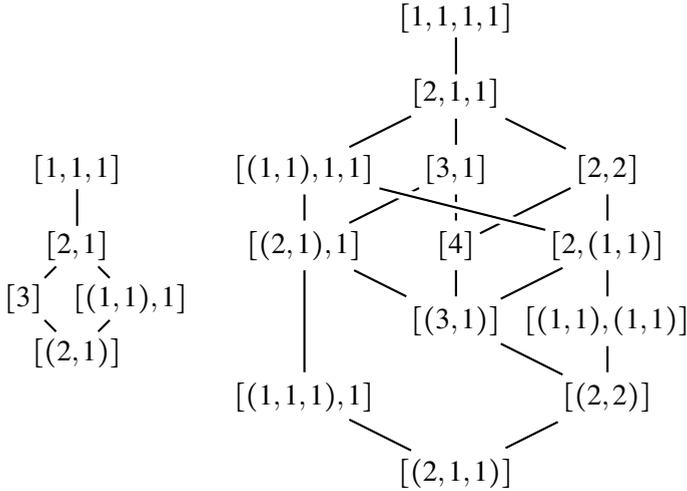
\begin{figure}[h]
\begin{center} 
 \begin{tikzpicture}
  \node (max) at (0, 5) {$[1, 1, 1, 1]$};
  \node (two) at (0,4) {$[2,1,1]$};
  \node (threel) at (-2,3) {$[(1, 1), 1, 1]$};
  \node (three) at (0,3) {$[3, 1]$};
  \node (threer) at (2,3) {$[2, 2]$};
  \node (fourl) at (-2,2) {$[(2, 1), 1]$};
  \node (four) at (0,2) {$[4]$};
  \node (fourr) at (2,2) {$[2, (1, 1)]$};
  \node (five) at (0,1) {$[(3, 1)]$};
  \node (fiver) at (2,1) {$[(1, 1),(1, 1)]$};
  \node (sixl) at (-2,0) {$[(1, 1, 1), 1]$};
  \node (sixr) at (2,0) {$[(2, 2)]$}; 
  \node (min) at (0, -1) {$[(2, 1, 1)]$};
  \draw [black,  thick, shorten <=-2pt, shorten >=-2pt] (min) -- (sixr) -- (fiver) -- (fourr) -- (threer) -- (two) -- (max);
  \draw [black,  thick, shorten <=-2pt, shorten >=-2pt](two) -- (three) -- (four) -- (five) -- (sixr);
  \draw [black,  thick, shorten <=-2pt, shorten >=-2pt](two) -- (threel);
  \draw [black,  thick, shorten <=-2pt, shorten >=-2pt](threel) -- (fourl) -- (three);
  \draw [black,  thick, shorten <=-2pt, shorten >=-2pt](threer) -- (four);
  \draw [black,  thick, shorten <=-2pt, shorten >=-2pt](fourl)--(five) -- (fourr);
  \draw [black,  thick, shorten <=-2pt, shorten >=-2pt](fourl)--(sixl) -- (min);
   \draw[preaction={draw=white, -,line width=4pt}] (threel) -- (fourr);
 \draw [black,  thick, shorten <=-2pt, shorten >=-2pt] (threel) -- (fourr);

 \node (top) at (-5,3) {$[1, 1, 1]$};
    \node [below of =top]  (next) {$[2, 1]$};
    \node [below left  of=next] (left)  {$[3]$};
    \node [below right of=next] (right) {$[(1,1),1]$};
    \node [below left of=right] (bottom) {$[(2,1)]$};

    \draw [black,  thick, shorten <=-2pt, shorten >=-2pt] (top) -- (next);
    \draw [black,  thick, shorten <=-2pt, shorten >=-2pt] (next) -- (left);
    \draw [black,  thick, shorten <=-2pt, shorten >=-2pt] (next) -- (right);
    \draw [black,  thick, shorten <=-2pt, shorten >=-2pt] (right) -- (bottom);
    \draw [black,  thick, shorten <=-2pt, shorten >=-2pt] (left) -- (bottom);

\end{tikzpicture}
\caption{The posets of all Segre symbols for $n=3$ (left) and $n=4$ (right).}
\label{fig:segre34}
\end{center}
\end{figure}

We wish to study the strata ${\rm Gr}_\sigma$ in (\ref{eq:stratum}).
Recall that ${\rm Gr}_\sigma$ is the constructible subset of 
${\rm Gr}(2,\SSS^n)$ whose points are the pencils $\mathcal{L}$
with $\sigma(\mathcal{L}) = \sigma$. Its closure
$\overline{{\rm Gr}}_\sigma$ is a subvariety of the Grassmannian ${\rm Gr}(2,\SSS^n)$.
Its defining equations can be written either in the $\frac{1}{8}(n+2)(n+1)n(n-1)$
Pl\"ucker coordinates, or in the $(n+1)n\,$ Stiefel coordinates
which are the matrix entries in a basis $\{A,B\}$ of $\mathcal{L}$.

Consider the related {\em Jordan stratification}.
 For each $\sigma \in {\rm Segre}_n$,
the Jordan stratum ${\rm Jo}_\sigma$ is the set of $n \times n$ matrices
whose Jordan canonical form has pattern~$\sigma$. Its closure
$\overline{{\rm Jo}}_\sigma$ is an affine variety in $\CC^{n \times n}$.
Its defining prime ideal consists of homogeneous polynomials
in the entries of an $n \times n$ matrix $X = (x_{ij})$.

\begin{theorem} \label{thm:poset} Our poset models inclusions of both
Grassmann strata and Jordan strata. That is,
 $\,\sigma \succeq \tau\,$ in ${\rm Segre}_n$ if and only if
$\,\,\overline{{\rm Gr}}_\sigma \supseteq \overline{{\rm Gr}}_\tau\,\,$  if and only if
$\,\,\overline{{\rm Jo}}_\sigma \supseteq \overline{{\rm Jo}}_\tau$.
 \end{theorem}

The codimensions of the Jordan strata generally  differ from those of the
Grassmann strata.
While the $\overline{{\rm Jo}}_\sigma $ are familiar from linear algebra \cite{DE},
the $ \overline{{\rm Gr}}_\sigma  $ capture the geometry of the
varieties listed on the right in  Examples \ref{ex:n3} and~\ref{ex:n4}.
The codimensions are $\geq 1$, unless $\sigma= [1,\ldots,1]$
where both strata are dense.

\begin{example}[$n=3$] \label{ex:strata3}
We computed the prime ideals for the Jordan strata in $\CC^{3 \times 3}$,
for the Pl\"ucker strata in ${\rm Gr}(2,\SSS^3) \subset \PP^{14}$,
and for the Stiefel strata in $\PP^5 {\times} \PP^5 $:
 $$
  \begin{small}
\begin{matrix} 
\hbox{symbol} & {\rm Jordan}  & \hbox{Pl\"ucker} & {\rm Stiefel} & \! \!\! {\rm codims}\! \! & {\rm degrees} 
  \\ [2,1]  & 6_1  & 6_1 &  (6,6)_1 & 1,1,1 &   6,6,[6,6] \\
[\,3\,]  & 2_1,3_1  & 4_{21} & \!\! (2,4)_1,(3,3)_1,(4,2)_1 \!\! & 2,2,2 & 6,99, [6,15,6] \\
[(1,\!1),1] & 3_{20} & 3_{20} & (3,3)_{20} &3,2,2 & 6,36,[4,4,4] \\
[(2,1)] & 2_9 & 2_6 & (2,2)_6 & 4,3,3 &  6,56,[4,12,12,4]
\end{matrix}
\end{small}
$$
The sextic in the first row is
 the discriminant of the characteristic polynomial of $X$.
 We shall explain the last row, indexed by $\sigma = [(2,1)]$.
 The Jordan stratum ${ \rm Jo}_\sigma $ has codimension $4$ and degree $6$.
 Its ideal is generated by nine quadrics, like
$ x_{11} x_{31}-2 x_{22} x_{31}+3 x_{21} x_{32} + x_{31} x_{33}$.
Under the substitution $X = AB^{-1}$, these transform into
six quadrics in Pl\"ucker coordinates, like
$ p_{04 } p_{14}+p_{12} p_{14}-p_{03} p_{15}-p_{12} p_{23}-3 p_{02} p_{34}+2 p_{01} p_{35}$.
 Here $p_{01},p_{02},\ldots,p_{45}$ denote the $2 \times 2$ minors of
 $$ \begin{small} \begin{pmatrix}
 a_{11} & a_{12} & a_{13} & a_{22} & a_{23} & a_{33} \\
 b_{11} & b_{12} & b_{13} & b_{22} & b_{23} & b_{33} \\
 \end{pmatrix} \end{small}.
 $$
 The stratum ${ \rm Gr}_\sigma $ has codimension $3$ in ${\rm Gr}(2,\SSS^3)$
 and degree $56$ in the ambient $\PP^{14}$.
 The six Pl\"ucker quadrics  give six polynomials of bidegree $(2,2)$ in
 $(A,B)$. These define a variety of multidegree
$4 a^3 + 12 a^2b + 12 ab^2 + 4 b^3 \,\in\, H^*( \PP^5 \times \PP^5)$.
\end{example}

\begin{example}[$n=4$]
The column ``codims'' in Example \ref{ex:n4}
gives the codimensions of Jordan strata, Pl\"ucker strata and Stiefel strata.
The last two agree; they quantify the moduli of quartic
curves in $\PP^3$  listed on the right.
We found equations of low degree for the $13$ strata.
For instance, ${\rm Jo}_{[4]}$  lies on a unique quadric:
$$ 
\begin{matrix} 3 x_{11}^2-2 x_{11} x_{22}-2x_{11} x_{33}- 2 x_{11} x_{44}
+8 x_{12}x_{21}+8x_{13} x_{31}+8 x_{14} x_{41} +3 x_{22}^2 \\ -2x_{22}x_{33}-2x_{22}x_{44}
+8 x_{23} x_{32}+8 x_{24} x_{42}+3 x_{33}^2-2 x_{33}x_{44}+8 x_{34} x_{43}+3 x_{44}^2. 
\end{matrix} 
$$
\end{example}

\begin{proof}[Proof of Theorem \ref{thm:poset}]
For Segre symbols $\sigma$ with one partition $\sigma_1$,
the Jordan strata ${\rm Jo}_\sigma$ are the {\em nilpotent orbits}
of Lie type $A_{n-1}$. Gerstenhaber's Theorem \cite{Ger} states that inclusion
of nilpotent orbit closures corresponds to the dominance order $ \vartriangleleft $
among the partitions $\sigma_1$. This explains the second condition
in our definition of ``is above'' for the poset ${\rm Segre}_n$.
The other condition captures the degeneration that occurs
when two eigenvalues come together. Generally, this leads to a 
fusion of Jordan blocks, made manifest by adding partitions
$\sigma_i$ and $\sigma_j$. For a precise algebraic version
of this argument
we refer to \cite[Theorem 4]{Ger}.

The inclusions of orbit closures
are preserved under the map  $X \mapsto A B^{-1}$ that links
Stiefel strata to Jordan strata. Furthermore, the Pl\"ucker
stratification is obtained from the Stiefel stratification
by taking the quotient modulo ${\rm GL}(2)$. This operation
also preserves the combinatorics of orbit closure inclusions.
\end{proof}

We close with formulas for the dimensions of our strata. For each partition 
$\sigma_i$ occurring in a Segre symbol $\sigma = [\sigma_1,\ldots,\sigma_r]$,
we write  $\sigma_i^* = (\sigma^*_{i1},\ldots,\sigma^*_{in}) $
for the conjugate partition. For instance, if $n=5$ and
$\sigma_i = (4,1)$ then $\sigma_i^* = (2,1,1,1)$.

\begin{proposition} \label{prop:codim}
The codimension of the Jordan strata  (in $\CC^{n    \times n}$)
and Grassmann strata (in ${\rm Gr}(2,\SSS^n)$)~are:
$$ {\rm codim} ({\rm Jo}_\sigma) \,\,\, = \,\,\,
\sum_{i=1}^r \sum_{j=1}^n (\sigma^*_{ij})^2 \,\,-\,\, r  \quad \,\,{\rm and} \quad\,\,
 {\rm codim} ({\rm Gr}_\sigma) \,\,\, = \,\,\,
\sum_{i=1}^r \sum_{j=1}^n \binom{\sigma^*_{ij}+1}{2} \,\,-\,\, r . $$
\end{proposition}

\begin{proof}
The dimension is the number $r$ of distinct eigenvalues plus the dimension
of the ${\rm GL}(n)$-orbit of the general matrix or pencil in the stratum of interest.
Thus, the codimension is the dimension of its stabilizer subgroup minus $r$.
The codimension for Grassmann strata agrees with 
the codimension for Stiefel strata, so we 
may consider pairs of matrices $(A,B)$ when
determining $ {\rm codim} ({\rm Gr}_\sigma)$.

The stabilizer on the left is found in \cite[Theorem 2.1]{DE} or
\cite[Proposition~8]{Ger}, using the identity $\sum_{k=1}^s (2k-1) =
s^2$. The stabilizer dimension on the right is calculated in
\cite[Corollary 2.2]{Dmy} for general symmetric matrix pencils. 
For regular pencils, the case studied here, the Kronecker canonical
form in \cite[eqn.~(2.4)]{Dmy} only has $H$-components. Thus the dimension formula in \cite{Dmy}
becomes $d_{A, B} =
d_H + d_{HH}$, where $d_H = 0$ and $d_{HH} = \sum_{i \leq i', \lambda_i =
  \lambda_{i'}} \min(h_i, h_{i'})$. In our notation, this~is
$$ \sum_{i \leq k, \alpha_i = \alpha_k} \!\! \min(e_i, e_k) \,\,=\,\,
\sum_{i=1}^r \sum_{k=1}^n k \sigma_{ik}  \,\,=\,\,
 \sum_{i=1}^r \sum_{k=1}^n \sum_{j=1}^{\sigma_{ik}} k
\,\, =\,\, \sum_{i=1}^r \sum_{j=1}^n \sum_{k=1}^{\sigma^*_{ij}} k
\,\, =\,\, \sum_{i=1}^r\sum_{j=1}^n \binom{\sigma^*_{ij} +1}{2}
.$$
In conclusion, our proof consists of specific pointers to the articles \cite{DE, Dmy, Ger}.
\end{proof}

\medskip
\bigskip

\noindent {\bf Acknowledgements.} We thank 
Orlando Marigliano and Tim Seynnaeve for helpful conversations.
Yelena Mandelshtam was supported by a US National Science Foundation
Graduate Research Fellowship under Grant DGE
1752814. Finally, we thank the anonymous referee for
  constructive comments, which helped to improve the paper.

\begin{small}

\end{small}

\end{document}